\documentclass{amsart}
\usepackage{amsmath,amsthm,amssymb,color,graphicx}

\newcommand{\R}{\mathbb R}

\newcommand{\N}{\mathbb{N}}
\newcommand{\Z}{\mathbb{Z}}
\newcommand{\Q}{\mathbb Q}
\newcommand{\G}{\mathcal G}
\newcommand{\spa}{{\rm span \,}}
\newcommand{\supp}{{\rm supp \,}}

\newtheorem{theorem}{Theorem}
\newtheorem{lemma}[theorem]{Lemma}
\newtheorem{proposition}[theorem]{Proposition}

\newtheorem{remark}[theorem]{Remark}
\newtheorem{example}{Example}

\begin{document}

\title{Time-frequency shift invariance  and the {A}malgan {B}alian {L}ow Theorem}
\author{Carlos Cabrelli, Ursula Molter, G\"otz E. Pfander}

\address{\textrm{(C. Cabrelli)}
Departamento de Matem\'atica,
Facultad de Ciencias Exac\-tas y Naturales,
Universidad de Buenos Aires, Ciudad Universitaria, Pabell\'on I,
1428 Buenos Aires, Argentina and
IMAS/CONICET, Consejo Nacional de Investigaciones
Cient\'ificas y T\'ecnicas, Argentina}
\email{cabrelli@dm.uba.ar}

\address{\textrm{(G. Pfander)}
Mathematics, Jacobs University, 28759 Bremen, Germany \& Mathematisch Geographische Fakult\"at, KU Eichst\"att, Germany. }
\email{g.pfander@jacobs-university.de}

\address{\textrm{(U. Molter)}
Departamento de Matem\'atica,
Facultad de Ciencias Exac\-tas y Naturales,
Universidad de Buenos Aires, Ciudad Universitaria, Pabell\'on I,
1428 Buenos Aires, Argentina and
IMAS/CONICET, Consejo Nacional de Investigaciones
Cient\'ificas y T\'ecnicas, Argentina}
\email{umolter@dm.uba.ar}

\thanks{The research of
C.~Cabrelli and U.~Molter are partially supported by
Grants  PICT 2011-0436 (ANPCyT), PIP 2008-398 (CONICET). The DAAD supported  G\"otz E. Pfander through the PROALAR grant 56033216. This research was carried out while G\"otz E. Pfander was John von Neumann visiting professor at Technical University Munich, Fall 2013/2014. He would like to thank the mathematics department   for their hospitality during his stay. 
}

\keywords{Balian-Low Theorem, additional shift invariance, Gabor frames, time-frequency analysis, Feichtinger algebra}

\begin{abstract}
We consider smoothness properties of the generator of a principal Gabor space on the real line which is invariant under some additional translation-modulation pair. We prove that if a Gabor system  on a lattice  with rational density  is a Riesz basis for its closed linear span,   and if the closed linear span, a Gabor space,  has any additional translation-modulation invariance, then its generator cannot  decay well in time and in frequency simultaneously.
\end{abstract}

\maketitle

\section{Introduction}

The Balian-Low Theorem, a key result in time-frequency analysis, expresses the fact that time-frequency concentration and non redundancy are essentially incompatible. Specifically,  if $\varphi\in L^2(\R)$, $\Lambda \subset \R^d$ is a lattice and  the system $(\varphi, \Lambda)= \{e^{2\pi i \eta x} \varphi(x - u): (u, \eta) \in \Lambda\}$ is a Riesz basis for $L^2(\R)$, then 
$\varphi$ satisfies
\begin{align}\label{eqn:uncertainty-product}
  \Big(\int (x-a)^2 |\varphi (x)|^2 \, dx \Big)\cdot \Big(\int (\omega-b)^2 |\widehat \varphi (\omega)|^2 \, d\omega \Big) =\infty, \quad a,b\in\R.
\end{align}
This theorem was originally stated independently by Balian \cite{Bal81} and Low \cite{Low85} for orthogonal systems, but both of their proofs contained a gap, which was later filled by Coifman et.~al \cite{Dau90} who also generalized it to Riesz bases. For general references on the Balian-Low Theorem we refer the reader to \cite{BHW95,Hei07}. In \cite{BHW95}, the authors also state and prove the so called  Amalgam Balian Low Theorem, which states that if    $(\varphi,\alpha\Z\times\beta\Z)$ is a Riesz basis for  $L^2(\R)$,  then $\varphi$ cannot belong to the Feichtinger algebra $S_0(\R)$, a class of functions decaying well in time and frequency. For a  definition of $S_0(\R)$ see \eqref{s0} below. Note that the  Amalgam Balian Low Theorem is seemingly weaker then Balian-Low Theorem, but is not implied by it. 

We define the unitary operators, translation $T_u:L^2(\R)\longrightarrow L^2(\R)$, $T_u f(x)=f(x-u)$, modulation $M_\eta:L^2(\R)\longrightarrow L^2(\R)$, $M_\eta f(x)=e^{2\pi i \eta x} f(x)$, and time-frequency shift  $\pi(u,\eta)=M_\eta T_u$, where $u\in\R$ and $\eta\in\widehat \R$, the dual group of $\R$ which is isomorphic to $\R$. 
For $\varphi\in L^2(\R)$ and a lattice $\Lambda= R\Z^2\subset \R\times \widehat \R$, $R\in\R^{2\times 2}$, with density $|\det R\, |^{-1}$ if $R$ is full rank and density 0 else, we define Gabor systems as $(\varphi,\Lambda)=  \{\pi(\lambda)\varphi\}_{\lambda\in\Lambda}$ and Gabor spaces as $\G(\varphi,\Lambda)=\overline{\spa\{\pi(\lambda)\varphi\}}$, where $\overline V$ is the closure of $V$ in $L^2(\R)$. For background on Gabor systems we refer to the monograph  \cite{gro01}.

This paper  addresses the  question whether there may exist a $\mu\in \R\times \widehat \R\setminus \Lambda$  with $\pi(\mu)\varphi \in \G(\varphi,\Lambda)$.  Equivalently, for $\Lambda'$ being a subgroup of $\R\times \widehat \R$ containing $\Lambda$, under which conditions on $\varphi$ is it possible that $\G(\varphi,\Lambda)=\G(\varphi,\Lambda')$?  
    
The case that $\mu,\Lambda\subseteq \R\times\{0\}$ is discussed at length in terms of shift-invariant spaces in the literature, see for example \cite{ACHKM10,ASW11, AKTW12}.  Since the Fourier transform is unitary, analogous results are implied for $\mu,\Lambda \in \{0\} \times\widehat \R$.  
 As we shall see in Remark~\ref{remark:onlyshift} in Example~\ref{ex:1}, the case $\mu \in \R\times\{0\}$ and $\pi(\mu)\varphi \in \G(\varphi,\alpha\Z\times \beta\Z)$ does not necessitate that $\pi(\mu)\varphi$ is in the shift-invariant space $\G(\varphi,\alpha\Z\times\{ 0\})$, so even the  case with $\mu\in \R \times \{0\} $ is not covered in the literature.
 
 On the other hand, the existing Balian Low type results for shift-invariant spaces only apply to  {\em principal} shift-invariant spaces, that is, spaces that can be generated by just one generator. 
Even though Gabor spaces are particular cases of shift-invariant spaces, except for the case $\Lambda=\alpha\Z \times \{0\}$, they are not principal 
shift-invariant spaces, so these results do not apply in the setting considered here.

To state our result, we recall that the {\em Feichtinger algebra} $S_0(\R)$ is defined by
\begin{equation}\label{s0}
	S_0(\R) =\left\{f\in L^2(\R): V f(t,\nu)= \int f(x) e^{-(x-t)^2}e^{2\pi i x \nu}\, dx \in L^1(t,\nu)\right\}.
\end{equation}
Note that $Vf(t,\nu) \in L^2(t,\nu) \cap  L^\infty (t,\nu)$ for all  $f\in L^2(\R)$ and the requirement $Vf(t,\nu) \in L^1(t,\nu)$ essentially necessitates $L^1$ decay of $f$ and of its Fourier transform $\widehat f$. For details on the Feichtinger algebra see \cite{Fei81,FZ98,gro01}.

We establish the following theorem.
\begin{theorem}\label{thm:maintheorem}
 If $(\varphi,\Lambda)$ is a Riesz basis for its closed linear span $\mathcal G(\varphi,\Lambda)$ with $\varphi\in S_0(\R)$ and the density of the lattice $\Lambda$ is rational, then    $\pi(u,\eta)\varphi \notin \mathcal G(\varphi,\Lambda)$ for all $(u,\eta)\notin \Lambda$.

In the case $\Lambda=\alpha\Z\times\beta\Z$, then the condition $\varphi\in S_0(\R)$ can be replaced with the weaker condition that $Z_\alpha\varphi(x,\omega)=\sum_{n\in\Z} f(x+n\alpha)e^{-2\pi i \omega  n\alpha}$ is continuous on $\R\times\widehat\R$. 
\end{theorem}

Theorem~\ref{thm:maintheorem} generalizes the Amalgam Balian Low Theorem stated above. Indeed, $(\varphi,\Lambda)$ being a Riesz basis for  $L^2(\R)$ implies that the density of $\Lambda$ equals $1$, that is, $(\alpha\beta)^{-1}=1\in\Q$ in case $\Lambda=\alpha\Z\times\beta\Z$, and $\mathcal G(\varphi,\Lambda)=L^2(\R)$  implies that $\pi(u,\eta)\varphi\in \mathcal G(\varphi,\Lambda)$ for all $(u,\eta)\in\R \times \widehat \R$, so Theorem~\ref{thm:maintheorem} implies that $\varphi\notin S_0(\R)$.    

\begin{remark}\rm
The question of whether the condition $\varphi \in S_0(\R)$ in Theorem~\ref{thm:maintheorem} can be replaced with having a finite uncertainty product \eqref{eqn:uncertainty-product}  is left for further exploration. Similarly, we do not discuss the case of $\Lambda$ having irrational density in this paper.

To generalize our proof of Theorem~\ref{thm:maintheorem} to a higher dimensional setting, that is, $\varphi\in L^2(\R^d)$ and  $\Lambda \subset  \R^{2d}$, requires a restriction to $\Lambda$ being a so-called symplectic lattices in order to use intertwining operators to reduce the general problem  to lattices of the form
$\alpha_1\Z\times \ldots \times \alpha_d \Z \times\beta_1\Z\times\ldots \times\beta_1\Z$ \cite{gro01}.
\end{remark}

Our investigation is motivated in part by the following.  In orthogonal frequency division multiplexing, short, OFDM, information in form of a coefficient sequence $\{c_{k,\ell}\}_{k\in\Z, \ell \in I}$ is transmitted through a channel using the signal 
$$F\{c_{k,\ell}\}=\sum_{k\in\Z}\sum_{\ell \in I} c_{k,\ell}T_{k\alpha}M_{\ell \beta}\varphi.$$ The  index set $I$  depends on the for transmission available frequency band and is therefore finite in most OFDM applications.  For $F$ to  be boundedly invertible, $\varphi$ is chosen so that $(\varphi,\alpha\Z\times \beta \Z)$ is a Riesz basis for its closed linear span. Moreover, to utilize a communications channel efficiently, it is beneficial to choose $\varphi$ with good decay in time and in frequency, that is, $\varphi\in S_0(\R)$, or better, $\varphi$ is a Schwartz class function.

Theorem~\ref{thm:maintheorem} then implies that under these conditions, $\pi(u,\eta)\varphi \notin \mathcal G(\varphi,\alpha\Z\times \beta \Z)$ whenever $(u,\eta)\notin \alpha\Z\times \beta \Z$.  Unfortunately, distortions that the signal undergoes are time-shifts (delays of the signal) in case of time-invariant channels, or time-frequency shifts in case of mobile, time-varying communications channels.   Theorem~\ref{thm:maintheorem} shows that we cannot choose $\varphi\in S_0(\R)$ so that the transmission space $\G(\varphi,\alpha\Z\times\beta\Z)$ is invariant under perturbations $\pi(u,\eta)$ for $(u,\eta)\notin \alpha\Z\times \beta \Z$.

In some cases, the leakage out of $\G(\varphi,\alpha\Z\times \beta \Z)$ can be used to identify an unknown channel operator $H$, in particular, if $H$ is well approximated by a single time-frequency shift $\pi(u,\eta)$ \cite{KP06,PW06b,Pfa13}. Unfortunately,  $\pi(u,\eta)\varphi\notin \mathcal G(\varphi, \Lambda)$ for all $(u,\eta)\notin \Lambda$  and $(\varphi, \Lambda)$ is a Riesz sequence for $G(\varphi, \Lambda)$ does not imply that there is no $f\in G(\varphi, \Lambda)$ with $\pi(u,\eta)f\in G(\varphi, \Lambda)$, so a receiver would not be able to know whether $\pi(u,\eta)f$ was transmitted through the identity operator, or $f$ was transmitted and then perturbed by the operator $\pi(u,\eta)$.\footnote{For example, if $g$ is a Gaussian, we have $( g, \mathbb Z \times \frac 3 2 \mathbb \Z)$ is a Riesz basis for $\G( g, \mathbb Z \times \frac 3 2 \mathbb \Z)$ since the density of  $\mathbb Z \times \frac 3 2 \mathbb \Z$ is $\frac 2 3<1$, see \cite{Hei07} and references therein. It is then not difficult to construct $f \neq 0$ such that $f,\pi(\frac 1 2,0)f\in \G( g, \mathbb Z \times \frac 3 2 \mathbb \Z)$.}
\subsubsection*{Related work}
Aldroubi, Sun and Wang   showed that if a principal shift-invariant space on the real line is also translation-invariant, that is, invariant under {\em every} translation operator, then any of
its  Riesz generators are non-integrable. Moreover, if the generator of the shift-invariant space is also invariant under the translate by $\frac{1}{n}$, $n\in \N\setminus \{1\}$, then  $\int |x|^{1+\epsilon} |\varphi (x)|^2\, dx =\infty$ for all $\epsilon>0$ \cite{ASW11}.

Gabardo and Han showed that if $\Lambda=\alpha\Z\times\beta\Z$ has integer density $(\alpha\beta)^{-1}\geq 2$ and $\mathcal G (\varphi,\alpha\Z\times\beta\Z)\neq L^2(\R)$, then 
\eqref{eqn:uncertainty-product} holds.  In the reciprocal case, they show that if $\alpha\beta \in \N\setminus\{1\}$ and $(\varphi,\alpha\Z\times\beta\Z)$ is not a Riesz system for its closed linear span, then again \eqref{eqn:uncertainty-product} holds. Note that both cases do not represent the generic case  \cite{GH04}

Gr\"ochenig, Han, Heil, and Kutyniok show that if $(\varphi,\Lambda)$ and $(\widetilde \varphi,\Lambda)$ are biorthogonal Riesz basis for $\mathcal G (g,\Lambda)$, then \eqref{eqn:uncertainty-product} holds for either $\varphi$ or $\widetilde \varphi$ \cite{GHHK02}.

For general Balian Low type results, we refer the reader to \cite{ BHW92, BHW95, BHW98, Bw94, DLL95, FG97, Jan08}.

\subsubsection*{Organisation of the paper}   In Section~\ref{section:Notation} we  discuss   our main tool, the Zak transform. We then proceed to prove  Theorem~\ref{thm:maintheorem}  in Section~\ref{section:proof}; and in Section~\ref{examples} we construct functions that generate Gabor spaces containing additional shifts of the generator.

\section{The Zak transform}\label{section:Notation}

The analysis offered below is based on the   Zak transform which is densely defined on $L^2(\R)$ by

\begin{align*}
 Z_\alpha f(x,\omega)&=\sum_{k\in\Z} f(x+\alpha k)\, e^{-2\pi i \alpha  k \omega}, \quad (x,\omega)\in \R\times \widehat \R,
\end{align*}
where $\alpha>0$. We write $Zf(x,\omega) = Z_1f(x,\omega).$
 
 It is easily observed that 
$$Z f(x+  n,\omega)=e^{2\pi i n  \omega}Z f(x,\omega), \quad Z f(x,\omega+m)=Z f(x,\omega),$$ 
in short, $Zf$ is quasiperiodic.   Not only does $Z f$ on $  [0,1]\times[0,1]$ fully describe $f$, but we have $  \|Z f\|_{L^2 ([0,1]\times[0,1])}=\|f\|_{L^2(R)}$, that is, $Z$ is a unitary map onto the space of quasiperiodic functions on $\R\times\widehat \R$ where the latter is equipped with the $L^2 ([0,1]\times[0,1])$ norm.

We shall utilize the fact that with $\pi(u,\eta)=M_\eta T_u$ we have
\begin{align*}
 \big(Z \pi(u,\eta)f\big) (x,\omega)& = \sum_{k\in\Z} \big(\pi(u,\eta)f\big)(x+ k)\, e^{-2\pi i k \omega }\\
 & = \sum_{k\in\Z} e^{2\pi i (x+ k)\eta }f(x+ k-u)\, e^{-2\pi i k  \omega}\\
 & = e^{2\pi i x\eta }\sum_{k\in\Z} f(x-u+ k)\, e^{-2\pi i k  (\omega-\eta)  }\\
 & = e^{2\pi i \eta x}Z f(x-u,\omega-\eta). 
 \end{align*}
 In particular, we have  for $k,\ell\in\Z$ that 
 \begin{align*}
 \big(Z  \pi(k,\ell)f\big) (x,\omega)=e^{2\pi i \ell x }Z f(x- k,\omega-\ell)=e^{2\pi i (\ell x +k \omega)}Z f(x,\omega), 
 \end{align*}
where we used the quasiperiodicity of the Zak transform.

Note that $S_0(\R)$ is invariant under the Fourier transform, so $\varphi \in S_0(\R)$ if and only if $\widehat \varphi \in S_0(\R)$.
The key property of $S_0(\R)$ that we use is that $\varphi \in S_0(\R)$ implies $Z\varphi$ continuous.  Indeed, if  $\varphi$ is in the Wiener Amalgam space 
$$W(C(\R),l^1(\Z))=\{f\in L^2(\R) \text{ continuous with }\ \sum_{k\in\Z}\|f\|_{L^\infty([k,k+1])}<\infty \}\supset S_0(\R),$$ then the sum defining the Zak transform converges uniformly, so the given continuity of $\varphi\in S_0(\R)$ implies continuity of its Zak transform.  Note that  $\varphi \in W(C(\R),l^1(\Z))$  is not necessary for the Zak transform to be continuous. In Theorem~\ref{thm:maintheorem} we therefore offer the two conditions  $\varphi\in S_0(\R)$ and, if $\Lambda=\alpha\Z \times \beta \Z$, more generally, the scaled Zak transform $Z_\alpha \varphi$, that is, the Zak transform adjusted to the   lattice $\alpha\Z \times \beta \Z$, is continuous.

\section{Proof of Theorem~\ref{thm:maintheorem}}\label{section:proof}

The proof is by contradiction. Let $\Lambda\in\R\times \widehat \R$ be a discrete subgroup of rational density. Assume there exists $\varphi \in S_0(\R)$,  such that $(\varphi,\Lambda)$ is a Riesz basis for its closed linear span $\G(\varphi,\Lambda)$, and assume further that there is an element $(u,\eta)\in\R\times \widehat \R\setminus \Lambda$  with $\pi(u,\eta)\varphi \in  \G(\varphi,\Lambda)$. 

\subsubsection*{Step 1. Without loss of generality $\Lambda=\frac 1 Q \Z \times P\Z$ with $P,Q\in\N$}

 Clearly, any generic full rank lattice $\Lambda$ of density $\frac P Q$ can be written as $\Lambda=A(\frac 1 Q\Z\times P\Z)$ with $A\in\R^{2\times 2}$, $\det A=1$. Since any $A\in\R^{2\times 2}$ with $\det A=1$ is element of the symplectic group, there exists  a so-called metaplectic operator $U=U(A)$ with $U^\ast \pi(\frac m Q,  n P) U =\pi(A(\frac m Q,nP)^T)$ \cite{gro01}. The metaplectic operator $U$ is unitary, hence, $(\varphi,\Lambda)$ is a Riesz basis for its closed linear span $\G(\varphi,\Lambda)$ if and only if $(U\varphi,\frac 1 Q \Z \times P\Z)$ is a Riesz basis for its closed linear span $\G(U\varphi,\frac 1 Q \Z \times P\Z)$. Moreover, $\pi(u,\eta)\varphi\in \G(\varphi,\Lambda)$ implies for some sequence $\{c_\lambda\} \in \ell^2(\Lambda)$
\begin{align*}U^\ast\pi(A^{-1}(u,\eta)^T)U\varphi &=
\pi(AA^{-1}(u,\eta)^T)\varphi= \pi(u,\eta)\varphi 
=\sum_{\lambda\in\Lambda}c_\lambda \pi(\lambda)\varphi \\
&=\sum_{m,n\in\Z}c_{m,n} \pi(A(\tfrac m Q,  n P)^T)\varphi =\sum_{m,n\in\Z}c_{m,n} U^\ast \pi(\tfrac m Q,  n P)U\varphi.
\end{align*}
We summarize that  with $(\widetilde u, \widetilde v)=A^{-1} (u,\eta)^T\notin \frac 1 Q \Z \times P\Z$ since $(u,\eta)\notin \Lambda$, and $\widetilde \varphi=U\varphi\in S_0(\R)$ by invariance of $S_0(\R)$ under metaplectic operators \cite{gro01}, we have  $\pi(\widetilde u,\widetilde \eta)\widetilde \varphi \in \G(\widetilde \varphi,\frac 1 Q \Z \times P\Z)$.

If ${\rm density}\, \Lambda =0$,  then we can increase $\Lambda$ to a full rank lattice, maintaining the property that $(\varphi,\Lambda)$ is a Riesz sequence for its closed linear span.  The argument above is then applicable.

\subsubsection*{Step 2. Without loss of generality, we can choose $u$ and $\eta$ to be rational.}

Clearly, this is equivalent to the existence of $R\in\N$ with $R\cdot(u,\eta)\in \frac 1 Q \Z \times P\Z$. 

We proceed by showing that if there exists  $(u,\eta)\in \R\times \widehat \R$ with $\pi(u,\eta)\varphi \in \G(\varphi,\frac 1 Q \Z \times P\Z)$, then exists also a rational pair $(\widetilde u,\widetilde \eta)\in \R\times \widehat \R$ with $\pi(\widetilde u, \widetilde \eta)\varphi \in \G(\varphi,\frac 1 Q \Z \times P\Z)$.

First, observe that $\pi(u,\eta)\varphi \in \G(\varphi,\frac 1 Q \Z \times P\Z)$ implies that $\G(\varphi,\frac 1 Q \Z \times P\Z)$ is invariant under both, $\pi(u,\eta)$ and $\pi(\frac m Q, nP)$, $m,n\in\Z$, and therefore,    $\G(\varphi,\frac 1 Q \Z \times P\Z)$ is invariant under $\pi(\lambda)$ where $\lambda$ is in the group $\widetilde \Lambda$ generated by $(u,\eta)$ and $\frac 1 Q \Z \times P\Z$. Moreover, we have $\pi(\lambda)\varphi \in \mathcal G(\varphi,\Lambda)$ for all  $\lambda\in {\rm closure}\, \widetilde \Lambda \subseteq \R\times \widehat \R$. 

If, $u$ is irrational, then ${\rm closure}\, \widetilde \Lambda$ contains $\R\times\{\eta\}$ and we  can replace $(u,\eta)\notin \frac 1 Q \Z \times P\Z$ by $(\widetilde u,\eta) \in {\rm closure}\, \widetilde \Lambda\setminus \frac 1 Q \Z \times P\Z$ with  $\widetilde u\in\Q$. With the same argument, we are able to  replace an irrational $\eta$ with a rational number $\widetilde \eta\notin P\Z$.  
 
\subsubsection*{Step 3. The case $Q=1$. }

Choose  $R\in\N$ with $(Ru,R\eta)\in \Z\times P \Z$. Set $M_2=Ru$ and $M_1=R\eta$, by increasing $R$ we can assume that $M_2\eta/2$ is an integer and $P$ divides $M_1$.

We have  $\pi(u,\eta)\varphi \in \G(\varphi,\Z\times P\Z)$ if and only if 
$$
e^{2\pi i \eta x}Z \varphi(x-u,\omega-\eta) =\big(Z \pi(u,\eta)\varphi\big) (x,\omega) 
 \in  Z\G(\varphi,\Z\times P\Z).$$
 But
 \begin{align*}
 Z\G(\varphi,\Z\times P\Z) & = \overline\spa\{Z \pi(\lambda)\varphi,\ \lambda \in \Z\times P\Z\} \\
 & \ \ =\overline\spa\{e^{2\pi i (P\ell x+k\omega)}Z\varphi(x,\omega),\ (k,\ell)\in \Z\times \Z\}.
\end{align*}
So $\pi(u,\eta)\varphi \in \G(\varphi,\Z\times P\Z)$ if and only if there exist a sequence $c=(c_{k,\ell})\in \ell^2(\Z^2)$ with \begin{align*}
 e^{2\pi i \eta x }Z \varphi(x-u,\omega-\eta)  &= \sum_{k,\ell\in\Z}c_{k,\ell}\, e^{2\pi i (P\ell x+k\omega)}Z\varphi(x,\omega)
 \\
 &= h(x,\omega) Z\varphi(x,\omega), \quad (x,\omega)\in \R\times \widehat \R,
 \end{align*} 
where
\begin{align*}
 h(x,\omega)=\sum_{k,\ell\in \Z}c_{k,\ell}\, e^{2\pi i (P\ell x+k\omega)}
\end{align*}
  is a locally $L^2$ function which is $1/P$ periodic in $x$ and 1 periodic in $\omega$. Note that the construction of $h$ is based on the assumption that $(\varphi,\Z\times P\Z)$ is a Riesz basis for its closed linear span. 
Hence
 \begin{align}\label{eqn:rightshift}
 Z \varphi(x,\omega)  &= e^{-2\pi i \eta  (x+u)} h(x+u,\omega+\eta) Z\varphi(x+u,\omega+\eta), \quad (x,\omega)\in \R\times \widehat \R. \end{align}

The above together with the quasiperiodicity of the Zak transform implies
\begin{align*} 
&Z\varphi (x,\omega) 
  = e^{-2\pi i \eta  (x+u)}\, h(x   +u,\omega  +\eta)  \, Z\varphi(x   +u,\omega +\eta) \\
  & = e^{-2\pi i \eta  (x+u)} \, h(x  +u,\omega +\eta)  \, e^{-2\pi i \eta  (x+2u)}\, h(x  +2u,\omega +2\eta)  \, Z\varphi(x  +2u,\omega +2\eta) \\
 & = \ldots =Z\varphi(x+Ru,\omega+R\eta)\ \exp\Big(-2\pi i \eta (Rx+u\sum_{r=1}^{R}r) \Big)  \,  \prod_{r=1}^{R} h(x+ru,\omega+r\eta)  \\
  & = e^{2\pi i M_2\omega}\, Z\varphi(x,\omega )\ \exp\Big(-2\pi i ( M_1 x+M_2\eta (R+1)/2) \Big)  \,  \prod_{r=1}^{R} h(x+ru,\omega+r\eta)  \\
  & =  \, e^{2\pi i (M_2\omega - M_1 x)}\, Z\varphi(x,\omega ) 
  \prod_{r=1}^{R} h(x+ru,\omega+r\eta)  
 ,  \quad (x,\omega)\in \R\times \widehat \R,
 \end{align*} 
 where we used that  $M_2\eta/2$ is an integer. 
   
Hence $h$ satisfies the  quasiperiodicity condition
 \begin{align}\label{eqn:product1}
 \prod_{r=1}^{R} h(x+ru,\omega+r\eta) = e^{2\pi i (M_1x - M_2\omega )}
   ,  \quad (x,\omega)\in \supp Z\varphi \subseteq \R\times \widehat \R. \end{align}

Equation \eqref{eqn:product1}  holds a-priori only on $\supp Z\varphi$, we shall now extend it to hold on all of $\R \times \widehat \R$ based on the assumption that   $(\varphi,\alpha\Z\times \beta\Z)$ is a Riesz sequence for its closed linear span.

Indeed, a standard periodization trick gives
\begin{align*}
  \int_\R \Big|\sum_{k,\ell\in \Z} d_{k,\ell}\pi(k,P\ell) \varphi (x)\Big|^2\, dx & = \int_0^1\int_0^1 \Big|\sum_{k,\ell\in \Z} d_{k,\ell}\,e^{2\pi i (P\ell t-k\omega)}Z \varphi (t,\nu)\Big|^2\, dt\, d\nu \\
  & = \int_0^1\int_0^{1/P} \Big|\sum_{k,\ell\in \Z} d_{k,\ell}\,e^{2\pi i (P\ell t-k\omega)}\Big| \sum_{p=0}^{P-1}\Big| Z \varphi (t-\tfrac p P,\nu)\Big|^2\, dt\, d\nu, 
\end{align*}
 and, hence, we have that 
$(\varphi,\Z\times P\Z)$ is  a Riesz sequence if and only if 
\begin{align*}
 A\leq \sum_{p=0}^{P-1}\Big| Z \varphi (x-\tfrac p P,\omega)\Big|^2 
\leq B,\quad a.e.\, (x,\omega),
\end{align*}
for some $0<A\leq B<\infty$.
So, for almost every $x_0,\omega_0$  exists  $p_0\in \{0,1,\ldots, P-1\}$ so that 
$$ Z\varphi(x_0-\frac  {p_0} P, \omega_0)\neq 0.$$ 
Using the computations above, we have
\begin{align*} 
Z\varphi (x_0-\frac  {p_0} P,\omega_0)&=  Z\varphi(x_0-\frac {p_0} P,\omega_0 ) e^{2\pi i (M_2\omega_0 - M_1(x_0-\frac {p_0} P))}\,  
  \prod_{r=1}^{R} h(x_0-\frac  {p_0} P+ru,\omega_0+r\eta) \\
  &=  Z\varphi(x_0-\frac  {p_0} P,\omega_0 ) e^{2\pi i (M_2\omega_0 - M_1x_0)}\,  
  \prod_{r=1}^{R} h(x_0+ru,\omega_0+r\eta),  
 \end{align*} 
 where we used the fact that $h$ is $\frac 1 P$ periodic in $x$ and $P$ divides $M_1$. 
As $Z\varphi (x_0-\frac  {p_0} P,\omega_0)\neq 0$, we  have indeed
 $$
  \prod_{r=1}^{R} h(x_0+ru,\omega_0+r\eta) = e^{2\pi i (M_1 x_0-M_2\omega_0 )}.
 $$
 As $(x_0,\omega_0)$ was chosen arbitrarily (a.e.), we conclude \eqref{eqn:product1} holds for almost every $(x_0,\omega_0)$.
 
Moreover, observe that \eqref{eqn:rightshift} implies that the zero set of $Z\varphi$ is $(u,\eta)$ periodic, hence if $(x_0, \omega_0)$ satisfies $Z\varphi(x_0 - \frac{p_0}{P}, \omega_0) \not= 0$, we have
 \begin{align*} 
0&\neq Z\varphi(x_0 - \frac{p_0}{P}, \omega_0) = Z\varphi (x_0-u-\frac  {p_0} P,\omega_0-\eta) \\ 
 & = e^{-2\pi i \eta  (x_0-u-\frac  {p_0} P+u)}\, h(x_0-u-\frac  {p_0} P   +u,\omega_0-\eta  +\eta)  \, Z\varphi(x_0-u-\frac  {p_0} P   +u,\omega_0 -\eta+\eta)  \\ 
 & = e^{-2\pi i \eta  (x_0-\frac  {p_0} P)}\, h(x_0,\omega_0)  \, Z\varphi(x_0-\frac  {p_0} P   ,\omega_0).
 \end{align*} 
 Solving for $h(x_0,\omega_0)$  implies that $h(x,\omega)$ is continuous 
 on   $\R \times \widehat \R$, and therefore \eqref{eqn:product1} holds on all of $\R \times \widehat \R$.

The proof of the case $Q=1$  is completed, by proving in Step 4 that a function $h$ as constructed above does not exist.  

\subsubsection*{Step 4. Periodicity vs quasiperiodicity and conclusion of the case $Q=1$} 
 Proposition~\ref{prop:key1} below is an extension of the simple fact that  if $h(x)$ is a function satisfying  $e^{2\pi i Mx}=\prod_{r=1}^R h(x+r0)=h(x)^R$, then $h(x)\neq 0$ for all $x$. If further $h$ is $\frac{1}{P}$-periodic, then 
$$
	h(x)=h(x+\tfrac 1 P)=e^{2\pi i \frac M R (x+ \tfrac 1 P)}=e^{2\pi i \frac M {RP} }\,h(x),
$$
 and, hence,  $RP$ divides $M$.

 \begin{proposition} \label{prop:key1} Let $P_1,P_2,R\in\N$, $M_1,M_2 \in\Z$, and $u,\eta\in\R$. If $h(x,\omega)$ is continuous on $\R \times \widehat \R$, $1/P_1$ periodic in $x$, $1/P_2$ periodic in $\omega$ and
 \begin{align}\label{eqn:productbadinprop}
 e^{2\pi i  (M_1 x +M_2 \omega)} = \prod_{r=0}^{R-1} h(x+r u,\omega + r\eta),  \quad (x,\omega)\in  \R\times \widehat \R,
 \end{align} 
 then   $RP_1$ divides $M_1$ and  $RP_2$ divides $M_2$.
\end{proposition}
Before giving a proof, let us first use Proposition~\ref{prop:key1} to conclude the proof of Theorem~\ref{thm:maintheorem} for $\Lambda=\Z\times P\Z$. 

Using all assumptions, we have established the existence of a continuous  $h(x,\omega)$  which satisfies \eqref{eqn:product1} and is $1/P$ periodic in $x$, and $1$-periodic in $\omega$. Therefore \eqref{eqn:productbadinprop} is satisfied with
$$ M_1=R\eta,\quad M_2=-Ru, \quad  P_1=P\quad \text{and} \quad P_2=1.$$ 
Then Proposition~\ref{prop:key1} implies $M_1/(RP_1)\in \Z$, that is, $\eta=M_1/R\in P\Z$, and  $u=-M_2/R\in \Z$.  
We conclude that $(u,\eta)\in \Lambda=\Z\times P\Z$, a contradiction.

\begin{proof}[Proof of Proposition~\ref{prop:key1}]
\ 

 We have 
 \begin{align*}
 M_1 x + M_2\omega =  \sum_{r=0}^{R-1}  \arg h(x+ru,\omega+r\eta) \mod 1,  \quad (x,\omega)\in  \R\times \widehat \R,
 \end{align*} 
 where by  continuity of $h$, we can choose $\arg h(x,w)$ to be continuous as well. (Note that this necessitates the values of $\arg h$ to be real numbers, not only values in $[0,1)$.)
 
For $x=\omega=0$, we have $\sum_{r=0}^{R-1}  \arg h(ru,r\eta) = p\in \Z$. 

As $\arg h(x,w)$ is continuous, we have 
$$
\sum_{r=0}^{R-1}  \arg h(x+ru,\omega+r\eta) =p+M_1x+M_2\omega, \quad x,\omega \in \R\times \widehat \R.
$$  
Indeed, by varying $\omega$ (or $x$) by a small value,  $\sum_{r=0}^{R-1}  \arg h(x+ru,\omega+r\eta)-M_1x-M_2\omega$  can only vary marginally and not jump by an integer value.  We conclude in particular that (for $x=1, \omega=0$ and $x=0, \omega=1$ respectively)
$$\sum_{r=0}^{R-1}  \arg h(1+ru,r\eta) =p+M_1, \quad\sum_{r=0}^{R-1}  \arg h(ru,1+r\eta) =p+M_2.$$

But, now, $\arg h(0,0)-\arg h(1/P_1,0)=q_1\in \Z$ by $1/P_1$ periodicity of $\arg h(x,\omega)$ in $x$.  
Similarly to before,  $\arg h(x ,\omega)-\arg h(x+1/P_1,\omega)$ is an integer, and, this time by continuity in $x$ and $\omega$, we must have $\arg h(x,\omega)-\arg h(x+1/P_1,\omega)= q_1$ for all $x,\omega \in\R\times \hat\R$. Hence,   $\arg h(x,\omega)-\arg h(x+1,\omega)= P_1q_1$.
Similarly, $\arg h(x,\omega)-\arg h(x,\omega+1)= P_2q_2$  for all $x,\omega \in\R\times \hat\R$ where $q_2\in\Z$.

We conclude
\begin{align*}
p= \sum_{r=0}^{R-1}  \arg h(ru,r\eta)  =\sum_{r=0}^{R-1} \big( \arg h(ru+1,r\eta) +P_1q_1 \big)= p+ M_1+RP_1q_1,
\end{align*}
and 
\begin{align*}
p= \sum_{r=0}^{R-1}  \arg h(ru,r\eta)  =\sum_{r=0}^{R-1} \big( \arg h(ru,r\eta+1) +P_2q_2 \big)= p+M_2+RP_2q_2,
\end{align*}
that is, $RP_1q_1+M_1=0=RP_2q_2+M_2$, and the conclusion follows since $q_1,q_2\in\Z$. 
\end{proof}

\begin{remark}\rm
  If we drop the assumption that $Z\varphi$ is continuous but maintain the assumption that $(\varphi,\Lambda)$ is a Riesz sequence, then the arguments above  allow to construct an $L^2$ function $h$ satisfying \eqref{eqn:product1} a.e. on $\R\times\widehat \R$.  Then, Proposition~\ref{prop:key1} implies that $h$ is discontinuous, so $h$ is neither a trigonometric polynomial nor an absolutely convergent Fourier series. We conclude that whenever $\pi(u,\eta)\varphi \in \G(\varphi,\Lambda)$, $(\varphi,\Lambda)$ is a Riesz sequence, and $(u,\eta)\notin \Lambda$, then $\pi(u,\eta)\varphi$ has a slowly convergent series expansion in $(\varphi,\alpha \Z \times \beta \Z)$.
\end{remark}

\subsubsection*{Step 5. The rational case $\frac PQ \notin \N$.}

We choose again  $R\in\N$ with $(Ru,R\eta)\in  \Z \times P \Z$.  Set $M_2=Ru$ and $M_1=R\eta$, by increasing $R$ we can assume that $M_2\eta/2$ is an integer and $P$ divides $M_1$. 

We have  $\pi(u,\eta)\varphi \in \G(\varphi,\tfrac 1 Q\Z\times P \Z)$ if and only if 
\begin{align*}
e^{2\pi i \eta x}Z \varphi(x-u,\omega-\eta) &=\big(Z \pi(u,\eta)\varphi\big) (x,\omega)  \in  Z\G(\varphi,\tfrac 1 Q\Z\times P \Z).
\end{align*}
But
\begin{align*}
 Z\G(\varphi,\tfrac 1 Q\Z\times P \Z) &\ \ = \overline\spa\{Z \pi(\lambda)\varphi,\ \lambda \in \tfrac 1 Q\Z\times P \Z \} \\
 & \ \ =\overline\spa\{e^{2\pi i   \ell P x}Z\varphi(x-\tfrac k Q,\omega-\ell P),\ (k,\ell)\in \Z\times \Z\}
 \\
 & \ \ =\overline\spa\{e^{2\pi i   \ell P x }Z\varphi(x-\tfrac k Q,\omega),\ (k,\ell)\in \Z\times \Z\}.
\end{align*}
That is, if and only if there exist a sequence $c=(c_{k,\ell})\in \ell^2(\Z^2)$ with 
\begin{align*}
 e^{2\pi i \eta x }Z \varphi(x-u,\omega-\eta)  
 &= \sum_{k,\ell\in\Z}c_{k,\ell}\, e^{2\pi i   \ell P x }Z\varphi(x-\tfrac k Q,\omega)
 \\
 &= \sum_{q=0}^{Q-1}\sum_{k,\ell\in\Z}c_{q+kQ,\ell}\, e^{2\pi i (\ell P  x+k\omega)}Z\varphi(x- \tfrac q Q,\omega )
 \\
  &= \sum_{q=0}^{Q-1} h_q(x,\omega)\, Z\varphi(x- \tfrac q Q,\omega ),
 \quad (x,\omega)\in \R\times \widehat \R,
 \end{align*} 
 that is 
 \begin{align}\label{eqn:rational}
 Z \varphi(x,\omega)  &= e^{-2\pi i \eta  (x+u)} \sum_{q=0}^{Q-1}\, h_q(x+u,\omega+\eta)\, Z\varphi(x+u - \tfrac q Q,\omega+\eta ), \quad (x,\omega)\in \R\times \widehat \R, \end{align}
 where
\begin{align*}
 h_q(x,\omega)=\sum_{k,\ell\in \Z}c_{q+kQ,\ell }\, e^{2\pi i (P\ell x+k\omega)}
\end{align*}
are  locally $L^2$ functions which are $1/P$ periodic in $x$ and 1 periodic in $\omega$. (Note that we can assume that all $h_q$ are locally in $L^2$, since $(\varphi,\Lambda)$ is a Riesz system.)

Following Zeevi and Zibulski (see \cite{KZZ04, ZZ97, ZZ93})we set
$$\mathcal Z \varphi(x,\omega) = \big(Z \varphi(x,\omega), Z \varphi(x-\tfrac 1 Q,\omega),
Z \varphi(x-\tfrac 2 Q,\omega),\ldots, Z \varphi(x-\tfrac {Q-1} Q,\omega)\big)^T, $$ but extend it quasiperiodically to an infinite vector $\mathcal Z^\circ \varphi(x,\omega)$, that is,  for $p=sQ+r$, $r\in \{0,1,\ldots, Q-1\}$, $s\in\Z$, we have
$$\mathcal Z^\circ_{p} \varphi(x,\omega)  
 	=  Z \varphi(x-\tfrac p Q,\omega)= e^{-2\pi i s \omega} \mathcal Z^\circ_r(x,\omega)= e^{-2\pi i s \omega} \mathcal Z_r(x,\omega).$$
The above translates then into 
 \begin{align*}
 \mathcal Z^\circ_{p} \varphi(x,\omega)  
 	& = e^{-2\pi i \eta  (x-\tfrac p Q+u)} \sum_{q=p}^{Q-1+p} h_{q-p}(x-\tfrac p Q+u,\omega+\eta)\,\mathcal Z^\circ_q\varphi(x+u  ,\omega+\eta ) 
	\end{align*}
 which leads to the biinfinite matrix equation 
\begin{align}\label{eqn:biinifinitematrixequation}
 \mathcal Z^\circ \varphi(x,\omega)  &= e^{-2\pi i \eta  (x+u)} H(x+u,\omega+\eta)\, \mathcal Z^\circ \varphi(x+u,\omega+\eta) \end{align}
 where
 \begin{align*}
 H_{pq} (x,\omega)&= e^{2\pi i \eta \tfrac p Q  }\, h_{q-p}(x-\tfrac p Q,\omega)\, \quad \text{if } q-p\in\{0,1,\ldots,Q-1\}\text{ and } 0 \text{ else}.\end{align*}

The above and quasiperiodicity of the Zak transform implies similarly as in the case $Q=1$ that
\begin{align*} 
 \mathcal Z^\circ\varphi (x,\omega) 
 & =\exp\Big(-2\pi i \eta (Rx+u\sum_{r=1}^{R}r \Big)
\cdot 
  \prod_{r=1}^{R} H(x+ru,\omega+r\eta)  \mathcal Z^\circ\varphi(x+Ru,\omega+R\eta)\\
   & =  \, e^{2\pi i (M_2\omega - M_1 x)}\,
  \prod_{r=1}^{R} H(x+ru,\omega+r\eta)   \mathcal Z^\circ\varphi(x,\omega ) 
 ,  \quad (x,\omega)\in \R\times \widehat \R. 
 \end{align*} 
 where we used as before that  $M_2\eta/2$ is an integer.
 
 Using the fact that $H(x,\omega)$ is $1/P$ periodic in $x$ and that $P$ divides $M_1$ we have in addition that 
 \begin{align*} 
 \mathcal Z^\circ\varphi (x+\frac p P,\omega)
 & =  \, e^{2\pi i (M_2\omega - M_1 x)}\,
  \prod_{r=1}^{R} H(x+ru,\omega+r\eta)   \mathcal Z^\circ\varphi(x+\frac p P,\omega ) 
 ,  \quad p=0,\ldots, P-1. 
 \end{align*} 
 
 Hence, for fixed $(x,\omega)$, we have \begin{equation}\label{quasi-per}
  e^{2\pi i (M_1x - M_2\omega )} I =
   \prod_{r=1}^{R} H(x+ru,\omega+r\eta), \quad a.e. \, (x,\omega)\in\R\times\R,
\end{equation}
for every quasiperiodic sequence in the span of  $\mathcal Z^\circ\varphi(x+\frac p P,\omega )$, $p=0,\ldots, P-1$.
The following lemma implies that \eqref{quasi-per} is an identity of operators on  $Q$-quasiperiodic sequences for a.e. $(x,\omega)$ .
\begin{lemma}
 If $\varphi \in S_0(\R)$ and $(\varphi,\frac 1 Q\Z \times P\Z)$ is a Riesz basis for its closed linear span, then $\mathcal Z^\circ\varphi(x+\frac p P,\omega )$, $p=0,\ldots , P-1$, spans the space of $Q$-quasiperiodic sequences for almost every $(x,\omega)\in\R\times\R$.
\end{lemma}
\begin{proof} For any 
 $d=(d_{k,\ell})\in \ell^2(\Z^2)$, we have
 $$
 	\|\{d_{k,\ell}\}\|_{\ell^2}\asymp \| \sum_{k,\ell\in\Z}\ d_{k,\ell} \ \pi(\tfrac k Q, \ell P)\varphi\|_{L^2(R)} = \| \sum_{k,\ell\in\Z}\ d_{k,\ell}\ Z \pi(\tfrac k Q, \ell P)\varphi\|_{L^2([0,1]\times[0,1])}.
 $$ 
We compute as above
\begin{align*}
\sum_{k,\ell\in\Z}\ d_{k,\ell}\ Z \pi(\tfrac k Q, \ell P)\varphi(x,\omega)
 &= \sum_{k,\ell\in\Z} \ d_{k,\ell}\ e^{2\pi i   \ell P x }Z\varphi(x-\tfrac k Q,\omega)
 \\
 &= \sum_{q=0}^{Q-1}\sum_{k,\ell\in\Z}\ d_{q+kQ,\ell}\, e^{2\pi i (\ell P  x+k\omega)}Z\varphi(x- \tfrac q Q,\omega )
 \\
  &= \sum_{q=0}^{Q-1} m_q(x,\omega)\, Z\varphi(x- \tfrac q Q,\omega ),
 \quad (x,\omega)\in \R\times \widehat \R.
 \end{align*} 
 We conclude that for some $A>0$ and all $m_0(x,\omega),\ldots, m_{Q-1}(x,\omega)$ that are $1$ periodic in $\omega$ and $1/P$ periodic in $x$, we have
\begin{align}
 A \|\{d_{k,\ell}\}\|^2_{\ell^2}&= A \sum_{q=0}^{Q-1}\|m_q\|^2_{L^2([0,1])} \leq \| \sum_{q=0}^{Q-1} m_q(x,\omega)\, Z\varphi(x- \tfrac q Q,\omega )\|^2_{L^2([0,1]\times[0,1])} \notag \\
 &=    \sum_{p=0}^{P-1} \int_0^{\frac 1 P} \int_0^1 \left| \sum_{q=0}^{Q-1} m_q(x-\tfrac p P,\omega)\, Z\varphi(x-\tfrac p P- \tfrac q Q,\omega )\right|^2 \, d\omega \,dx \notag \\
 &=    \sum_{p=0}^{P-1} \int_0^{\frac 1 P} \int_0^1 \left| \sum_{q=0}^{Q-1} m_q(x ,\omega)\, Z\varphi(x-\tfrac p P- \tfrac q Q,\omega )\right|^2 \, d\omega \,dx \label{eqn:linearindependence}\\
 &\leq   \int_0^{\frac 1 P} \int_0^1   \sum_{q=0}^{Q-1} |m_q(x ,\omega)|^2\, \sum_{p=0}^{P-1} \left| Z\varphi(x-\tfrac p P- \tfrac q Q,\omega )\right|^2\, d\omega \,dx .\label{eqn:lowerbound}
\end{align}
From \eqref{eqn:lowerbound} we conclude that for  $q=0,\ldots, Q-1$ we have
\begin{align*}
 A\leq  \sum_{p=0}^{P-1} \left| Z\varphi(x-\tfrac p P- \tfrac q Q,\omega )\right|^2\quad a.e.\, (x,\omega)\in [0,1/P]\times[0,1].
\end{align*}
As  $\sum_{p=0}^{P-1} \left| Z\varphi(x-\tfrac p P- \tfrac q Q,\omega )\right|^2$ is $1/P$ periodic in $x$, this inequality holds in fact for a.e. $(x,\omega)\in \R\times\R$.
Moreover, \eqref{eqn:linearindependence} implies that for $q=0,\ldots,
Q-1$, the $\mathbb C^P$ vectors
\begin{align*}
 \left(Z\varphi(x-\tfrac q Q,\omega), Z\varphi(x-\tfrac q Q -\tfrac 1 P,\omega) , Z\varphi(x-\tfrac q Q -\tfrac 2 P,\omega), \ldots,  Z\varphi(x-\tfrac q Q -\tfrac {P-1} P,\omega)\right) 
\end{align*}
are linearly independent for a.e. $(x,\omega)\in [0,1/P]\times[0,1]$, indeed, else we could find $L^2(\R^2)$ functions $m_q(x,\omega)$, not all $m_q(x,\omega)=0$, such that \eqref{eqn:linearindependence} equals $0$. We conclude that the matrix
\begin{align*}
 \begin{pmatrix}
 Z\varphi(x ,\omega)& Z\varphi(x  {-}\tfrac 1 P,\omega) & Z\varphi(x  {-}\tfrac 2 P,\omega)&   {\ldots}&  Z\varphi(x  {-}\tfrac {P{-}1} P,\omega) \\
  Z\varphi(x{-}\tfrac 1 Q,\omega)& Z\varphi(x{-}\tfrac 1 Q {-}\tfrac 1 P,\omega) & Z\varphi(x{-}\tfrac 1 Q {-}\tfrac 2 P,\omega)&   {\ldots}&  Z\varphi(x{-}\tfrac 1 Q {-}\tfrac {P{-}1} P,\omega) \\
 Z\varphi(x{-}\tfrac 2 Q,\omega)& Z\varphi(x{-}\tfrac 2 Q {-}\tfrac 1 P,\omega) & Z\varphi(x{-}\tfrac 2 Q {-}\tfrac 2 P,\omega)&   {\ldots}&  Z\varphi(x{-}\tfrac 2 Q {-}\tfrac {P{-}1} P,\omega) \\
 \vdots &\vdots &\vdots & &\vdots  \\
  Z\varphi(x{-}\tfrac {Q{-}1} Q,\omega)& Z\varphi(x{-}\tfrac   {Q{-}1}Q {-}\tfrac 1 P,\omega) & Z\varphi(x{-}\tfrac  {Q{-}1} Q {-}\tfrac 2 P,\omega)&  {\ldots}&  Z\varphi(x{-}\tfrac  {Q{-}1} Q {-}\tfrac {P{-}1} P,\omega) \\
\end{pmatrix}
\end{align*}
is full rank for a.e. $(x,\omega)\in [0,1/P]\times[0,1]$, so its $P$ columns are a spanning set of $\mathbb C^Q$ for a.e. $(x,\omega)\in [0,1/P]\times[0,1]$. Note that replacing $x$ by $x-\tfrac {p_0} P$ in the matrix above corresponds to a circular shift of  the columns of the matrix by $p_0$, with the possible appearance of a non-zero scalar factor $e^{2\pi i \omega}$  due to the quasiperiodicity of the Zak transform.  This  allows us to extend the observation on the columns spanning $\mathbb C^Q$ to hold for almost every $(x,\omega)\in \R\times \widehat\R$.

\end{proof}
  
  In the $Q$ dimensional model, that is, choosing $\widetilde H(x,\omega)\in\mathbb C^{Q\times Q}$ so that for any $\mathcal Z\in\mathbb C^Q$ and any $(x,\omega)\in \R\times \widehat \R$ we have 
  $$ \Big(H(x,\omega) \mathcal Z^\circ\Big)_p= \Big(\widetilde H(x,\omega) \mathcal Z \Big)_p, \quad p=0,1,\ldots, Q-1,$$
  we have equivalently (with $I$ now denoting the identity matrix in $\mathbb C^{Q\times Q}$
   \begin{align*}
 e^{2\pi i (M_1x - M_2\omega )} I=
   \prod_{r=1}^{R} \widetilde H(x+ru,\omega+r\eta) ,  \quad a.e. \, (x,\omega)\in   \R\times \widehat \R. \end{align*}
Taking $h(x,\omega)=\det  \widetilde H(x,\omega)$ we conclude 
    \begin{align*}
 e^{2\pi i Q(M_1x - M_2\omega )} =
   \prod_{r=1}^{R} h(x+ru,\omega+r\eta) ,  \quad a.e. \, (x,\omega)\in   \R\times \widehat \R. \end{align*}
   It remains to argue that $h(x,\omega)$ is continuous, since then,  Proposition~\ref{prop:key1} and the $1/P$ periodicity of $h(x,\omega)$ in $x$ and the 1 periodicity in $\omega$ implies first that $R$ divides $QM_2$.  Hence $RL=QM_2$ for some $L\in\N$ and $u=M_2/R=L/Q\in\tfrac 1 Q \Z$.  Second, we have  $RP$ divides $QM_1$, that is, $\eta \tfrac Q P = \tfrac {QM_1}{RP}$ is an integer. By assumption, we have that $(P,Q)=1$, so $\eta\in P\Z$.   However, since by assumption $(u,\eta) \not\in \frac{1}{Q}\Z\times P\Z$, this is a contradiction.

We conclude by showing that $\widetilde H$ and therefore $h$ depends continuously on $(x,\omega)$.
To this end, observe that $\varphi\in S_0(\R)$ implies that   both, $\mathcal Z^\circ \varphi(x,\omega)$ and  $\mathcal Z\varphi(x,\omega)$   are continuous in $(x,\omega)$. 
Let $\Phi(x,\omega)\in\mathbb C^{Q\times P}$ be the frame synthesis matrix with columns $\mathcal Z \varphi(x+\frac p P,\omega )$, $p=0,\ldots, P-1$. Equation \eqref{eqn:biinifinitematrixequation} implies that 
\begin{align*}
   e^{2\pi i \eta  x} \mathcal Z \varphi(x-u,\omega-\eta)   &= \widetilde H(x,\omega)\, \mathcal Z \varphi(x,\omega).
\end{align*}
Inserting $x+\frac p P$ for $x$ and using that
 $\widetilde H(x,\omega)$ is $1/P$ periodic in $x$, we obtain 
\begin{align*}
   e^{2\pi i \eta  x} \ \Phi(x-u,\omega-\eta) \ D(\eta)  &= \widetilde H(x,\omega)\, \Phi(x,\omega), 
\end{align*}
where $D(\eta)$ is the diagonal matrix with entries $1,e^{-2\pi i \eta/P},e^{-2\pi i \eta \,2/P},\ldots,e^{-2\pi i \eta\,(P-1)/P} $.

The columns of $\Phi(x,\omega)$ form a frame that depends continuously on $(x,\omega)$. Hence, the  rame operator $S(x,\omega)=\Phi(x,\omega)\Phi(x,\omega)^\ast \in\mathbb C^{Q\times Q}$ and its inverse $S(x,\omega)^{-1}$ depend continuously on $(x,\omega)$. Similarly, the matrix consisting of the dual frame elements $\Psi(x,\omega)= S(x,\omega)^{-1}\Phi(x,\omega)$ depends continuously on $(x,\omega)$. Clearly,  $\Psi(x,\omega)^\ast$ is a right inverse of $\Phi(x,\omega)$, The equality 
\begin{align*}
   e^{2\pi i \eta  x} \ \Phi(x-u,\omega-\eta) \ D(\eta)\ \Psi(x,\omega)^\ast&= \widetilde H(x,\omega)\, \Phi(x,\omega) \, \Psi(x,\omega)^\ast=\widetilde H(x,\omega)
\end{align*}
shows that $\widetilde H(x,\omega)$ depends continuously on $(x,\omega)$. The proof is complete.

\section{Construction of Gabor spaces with additional shift invariance}\label{examples}

In this section, we  study the 
case $\pi(\frac 1 R,0)\varphi \in \G(\varphi,\Z\times P\Z)$, $\text{gcd}(P,R)=1$, and give a complete characterization of those $\varphi$ which satisfy  $\pi(\frac 1 R,0)\varphi \in \G(\varphi,\Z\times P\Z)$.

Recall that $\pi(\frac 1 R,0)\varphi \in \G(\varphi,\Z\times P\Z)$ if and only if  there exists a sequence $c=(c_{k,\ell})\in \ell^2(\Z^2)$ with 
\begin{align}\notag
 Z\varphi (x-\tfrac 1 R,\omega) &= \sum_{k,\ell\in\Z}c_{k,\ell}\, e^{2\pi i (P\ell x+k\omega)}Z\varphi(x,\omega)
 \\
 &= h(x,\omega) Z\varphi(x,\omega), \quad (x,\omega)\in \R\times \widehat \R. \label{eqn:key} \end{align}

Our strategy is to construct a quasiperiodic function $F(x,\omega)$ and a function $h(x,\omega)$ so that \eqref{eqn:key} holds with $F$ in place of $Z\varphi$.  Then we use the fact that the Zak transform is onto the space of quasiperiodic functions, and,  using   a Zak transform inversion formula \cite{gro01}.
we construct
\begin{align*}
 \varphi(x)=\int_0^1 Z\varphi(x,\omega)\, d\omega=\int_0^1 F(x,\omega)\, d\omega=\int_u^{1+u}F(x,\omega), \quad \text{a.e.}\  x\in\R.
\end{align*}

In order to construct the quasiperiodic function $F(x,\omega)$, we  shall show that the conditions
\begin{enumerate}
\item[(S)] $\displaystyle F (x-\tfrac 1 R,\omega)  = h(x,\omega) F(x,\omega), \quad x \in [1/R, 1],\ \omega\in [0,1];$
\item[(Q)] $\displaystyle e^{2\pi i  \omega} =\prod_{r=0}^{R-1} h(x+\tfrac r R,\omega), \quad (x,\omega) \in [0,1/P]{\times}[0,1]\cap \supp F; $
\item[(P)] $h(x,\omega)$ is $1/P$ periodic in $x$ and 1 periodic in $\omega$,
 \end{enumerate}
characterize the pairs $F(x,\omega)=Z\varphi(x,\omega)$ and $h(x,\omega)$ that satisfy \eqref{eqn:key}.

First, note that \eqref{eqn:key} implies that the zero set $\mathcal E$ of $Z\varphi$ is $1/R$ periodic. Indeed, clearly $ Z\varphi(x,\omega)=0$ implies  $ Z\varphi(x-\tfrac 1 R,\omega)=0$. But also, $Z\varphi(x,\omega)=0$ implies
\begin{align*}
 0=Z\varphi(x+1,\omega)=Z\varphi(x+1-\tfrac 1 R,\omega)=Z\varphi(x+1-\tfrac 2 R,\omega)=\ldots=Z\varphi(x+\tfrac 1 R,\omega).
\end{align*}

In addition, since $R\in\N$,  the quasiperiodicity conditions 
 \begin{align*}
 e^{2\pi i M \omega} =\prod_{r=1}^{RM} h(x+\tfrac r R,\omega),  \quad (x,\omega)\in \supp Z\varphi \subseteq \R\times \widehat \R, \end{align*} 
  are just the $M$-th power of the equation where $M=1$, that is, the set of equations is equivalent to 
  \begin{align*}
 e^{2\pi i  \omega} =\prod_{r=1}^{R} h(x+\tfrac r R,\omega),  \quad (x,\omega)\in \supp Z\varphi \subseteq \R\times \widehat \R. \end{align*} 

We conclude that the three conditions given above follow from \eqref{eqn:key}. To observe that these conditions are also sufficient, note first that as  argued above, quasiperiodicity of $F$ implies that the zero set of $F$ is $1/R$ periodic.  Hence, condition (b) extends to all $(x,\omega)\in \R\times \widehat \R$.  Now, (b) together with (a) implies that 
\begin{align*}
  h(x,\omega)F(x,\omega)=F(x-\tfrac 1 R,\omega)=e^{-2\pi i \omega} F(x+\tfrac{R-1}{R},\omega),\quad x\in [0,1/R]{\times}[0,1].
\end{align*}
Indeed, it suffices to check this on $\supp F$ where we have \begin{align*}
 e^{-2\pi i \omega} F(x+\tfrac{R-1}{R},\omega) 
 &= \prod_{r=0}^{R-1} h(x+\tfrac r R,\omega) F(x+\tfrac{R-1}{R},\omega)\\
 & \vdots \\
  &= h(x,\omega) h(x+\tfrac 1 R,\omega) F(x+\tfrac{1}{R},\omega)\\
  &= h(x,\omega) F(x,\omega),
\end{align*}
which concludes our proof of sufficiency.

\begin{figure}
 \includegraphics[scale=0.3]{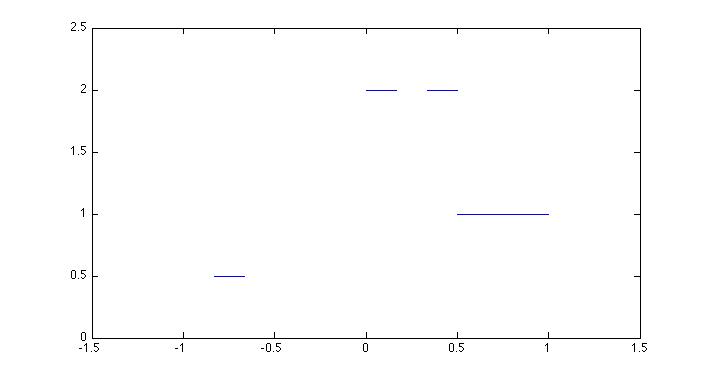}\includegraphics[scale=0.3]{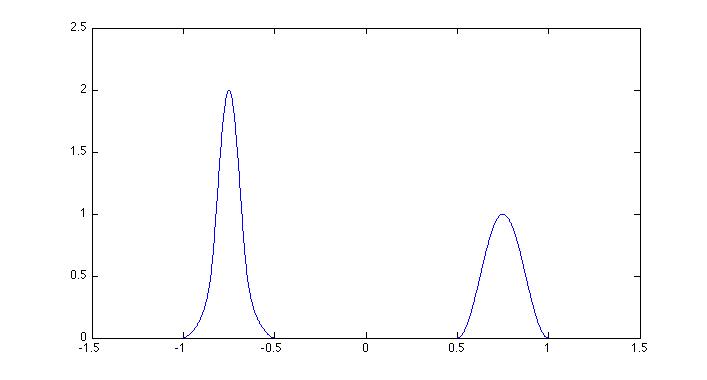}
 \caption{Functions as constructed in Example~\ref{ex:1} and Example~\ref{ex:2}}
\end{figure}

In our  first example, we construct a discontinuous window function which generates a Gabor space that features an additional shift invariance. 

\begin{example}\label{ex:1}\rm
 We choose $R=2$ and $P=3$. Let  $I_k=[k/6,(k+1)/6]\times \widehat \R$ for $k\in \Z$. We define the function $h(x,\omega)=2$ on $\bigcup_k I_{2k}$ and $h(x,\omega)=e^{2\pi i \omega}/2$ on $\bigcup_k I_{2k+1}$. Clearly, $h$ satisfies $(Q)$ and $(P)$. We set $F(x,\omega)=1$ for $x\in[1/2,1]$ and 
$$F(x,\omega)=F(x+1/2-1/2,\omega)=h(x+1/2,\omega)F(x+1/2,\omega)=h(x+1/2,\omega),\quad x\in[0,1/2],$$
and extend the function quasiperiodically.  In the following, let $\int=\int_{-1/2}^{1/2}$.  Motivated by the Zak transform inversion formula, we define for $x\in\R$,
\begin{align*}
 \varphi(x)&=\int F(x,\omega)\, d\omega= 
\begin{cases}
\int e^{2\pi i m \omega} =\delta_0(m),&\ x\in [m+1/2,m+1],\\
\int 2e^{2\pi i m \omega} =2\delta_0(m),&\ x\in [m,m{+}1/6]{\cup} [m{+}1/3, m{+}1/2],\\
\int \frac 1 2 e^{2\pi i (m+1) \omega} =\frac 1 2 \delta_0(m+1),&\ x\in [m+1/6,m+1/3], \\
\end{cases}\\
&= 1/2 \chi_{[-5/6,-2/3]} +2\chi_{[0,1/6]} + 2\chi_{[1/3,1/2]}+\chi_{[1/2,1]}.
\end{align*}

Clearly, $\varphi\notin S_0(\R)$.  Moreover, note that $9/4 \leq \sum_{p=0}^{2}\Big| Z_\varphi (t-\tfrac p 3,\nu)\Big|^2\leq 9$ for all $t$ and $\nu$ implies that  $(\varphi,\Z\times P\Z)$ is a Riesz basis for $\G(\varphi,\Z\times P\Z)$.

In addition, we would like to point out once more that $h(t,\nu)=\sum c_{k,\ell} e^{2\pi i (P\ell x-k\nu)}$ not being continuous implies that $\pi(1/2,0)\varphi =\sum c_{k,\ell}\pi(k,P\ell)\varphi$ converges rather slowly, for example, we do not have absolute convergence. 

\begin{remark}\label{remark:onlyshift}\rm
 Note that the shift-invariant space $\G(\varphi,\Z\times \{ 0\})$ is constant on the intervals $[m+1/2,m+1]$, $m\in \Z$, but the half shift $\varphi(x-1/2)$ does not satisfy this property. Hence,  $\pi(1/2,0)\varphi=T_{1/2}\varphi$ is not a member of the shift-invariant space $\G(\varphi,\Z\times \{ 0\})$, showing that membership of translates to Gabor spaces cannot be reduced to membership of translates to respective shift-invariant spaces.
\end{remark}

\end{example}
In the following, we construct a smooth window $\varphi$ which has an additional shift invariance and which generates therefore not a Riesz basis for the Gabor space it spans.  Note that   mollifying $h$ in the example above leads to a a continuous function which does not satisfy property  (Q).  

\begin{example}\label{ex:2} \rm  We consider again $R=2$ and $P=3$ and construct a Schwartz class function $\varphi$ such that
$T_{\frac 1 2}\varphi \in \mathcal G(\varphi, \Z\times 3\Z)$.

To this end, choose a function $u(x)$ on $[0,1/2]$ with 
\begin{enumerate}
\item $u$ has only values in $[\frac 1 2,2]$, $u(0)=1$ but $u$ not constant 1;
\item $u$ is smooth;
\item $u(x)u(x+1/6)=1$ for $x\in[0,\frac 1 3]$.
\end{enumerate}
Now, set $h(x,\omega)=u(x)$ for $x\in [0, \frac 1 2]$ and $h(x,\omega)=e^{2\pi i w}/u(x-1/2)$ for $x\in [\frac 1 2,2]$. So $h$ periodically extended is smooth away from the set $\frac 1 2 \Z \times \widehat \R$ and satisfies (Q).

Now, we define $F(x,\omega)=v(x)$ for $x\in[1/2,1]$ where $v(1/2)=v(1)=0$, $v(x)\in [0,1]$, and $v$ smooth.  Further, define
\begin{align*}
  F(x,\omega)&=F(x+1/2 -1/2,\omega)=h(x+1/2,\omega)F(x+1/2,\omega)\\ &=e^{2\pi i \omega}v(x+1/2)/u(x),\quad x\in[0,1/2].
\end{align*}  Clearly, $F$ is smooth away from  $\frac 1 2 \Z \times \widehat \R$, but by choosing $v^{(n)}(0)=v^{(n)}(1/2)=0$ for all $n\in\N$ ensures that $F$ is smooth on $\R\times \widehat \R$.

We compute
\begin{align*}
 \varphi(x)&=\int_{-1/2}^{1/2} F(x,\omega)\, d\omega \\
 & = 
\begin{cases}
\int e^{2\pi i m \omega} v(x) =\delta_0(m) v(x),&\ x\in [m+1/2,m+1],\\
\int \frac {v(x+1/2)}{u(x)} e^{2\pi i (m+1) \omega} =\frac {v(x+1/2)}{u(x)} \delta_0(m+1),&\ x\in [m ,m+1/2], \\
\end{cases}\\
& = 
\begin{cases}
 \frac {v(x+1/2)}{u(x)}  ,&\ x\in [-1 ,-1/2], \\
   v(x),&\ x\in [1/2,1],\\
\end{cases}\\
\end{align*}
We conclude that $\varphi$ is supported on $[-1,1]$ and smooth.   \end{example}
\newcommand{\etalchar}[1]{$^{#1}$}
\def\cprime{$'$} \def\cprime{$'$}
\providecommand{\bysame}{\leavevmode\hbox to3em{\hrulefill}\thinspace}
\providecommand{\MR}{\relax\ifhmode\unskip\space\fi MR }
\providecommand{\MRhref}[2]{%
  \href{http://www.ams.org/mathscinet-getitem?mr=#1}{#2}
}
\providecommand{\href}[2]{#2}

%
\end{document}